\documentclass{amsart}
\usepackage{amssymb}
\usepackage{stmaryrd} 
\usepackage{amsmath} 
\usepackage{amscd}
\usepackage{amsthm}
\usepackage{amsbsy}
\usepackage[margin=1.2 5in]{geometry}
\usepackage{commath, longtable}
\usepackage{comment, enumerate}
\usepackage{geometry}
\usepackage[matrix,arrow]{xy}
\usepackage{hyperref}
\usepackage{mathrsfs}
\usepackage{color}
\usepackage{float}
\usepackage{mathtools,caption}
\usepackage[table,dvipsnames]{xcolor}
\usepackage{tikz-cd}
\usepackage{longtable}
\usepackage[utf8]{inputenc}
\usepackage[OT2,T1]{fontenc}
\usepackage[normalem]{ulem}
\usepackage{hyperref}
\usepackage[customcolors]{hf-tikz}
\usepackage{enumitem}
\newlist{primenumerate}{enumerate}{1}
\setlist[primenumerate,1]{label={\arabic*$'$}}
\hypersetup{
 colorlinks=true,
 linkcolor=DarkOrchid,
 filecolor=blue,
 citecolor=olive,
 urlcolor=orange,
 pdftitle={GHKL project},
 }
\usepackage{booktabs}

\DeclareSymbolFont{cyrletters}{OT2}{wncyr}{m}{n}
\DeclareMathSymbol{\Sha}{\mathalpha}{cyrletters}{"58}

\newtheorem{theorem}{Theorem}[section]
\newtheorem{lemma}[theorem]{Lemma}

\newtheorem{proposition}[theorem]{Proposition}

\newtheorem{definition}[theorem]{Definition}

\newtheorem*{notn}{Notation}
\newtheorem*{theorem*}{Theorem}
\numberwithin{equation}{section}
\newtheorem{lthm}{Theorem}

\theoremstyle{remark}
\newtheorem{remark}[theorem]{Remark}

\newcommand{\EC}{\mathsf{E}}

\newcommand{\fp}{\mathfrak{p}}

\newcommand{\Gal}{\operatorname{Gal}}

\newcommand{\PP}{\mathbb{P}}
\newcommand{\Qp}{\mathbb{Q}_p}
\newcommand{\Zp}{\mathbb{Z}_p}

\newcommand{\ord}{\mathrm{ord}}
\newcommand{\GL}{\mathrm{GL}}

\newcommand{\Z}{\mathbb{Z}}
\newcommand{\Q}{\mathbb{Q}}

\newcommand{\cL}{\mathcal{L}}

\newcommand{\cK}{\mathcal{K}}
\newcommand{\cO}{\mathcal{O}}

\newcommand{\Sel}{\mathrm{Sel}}
\newcommand{\Char}{\mathrm{char}}

\newcommand{\corank}{\mathrm{corank}}

\newcommand{\fP}{\mathfrak{P}}

\newcommand{\cyc}{\mathrm{cyc}}

\newcommand{\ac}{{\mathrm{ac}}}
\newcommand{\ann}{{\mathrm{ann}}}

\makeatletter
\theoremstyle{plain} 
\newtheorem*{intr@thm}{\intr@thmname}

\newtheorem*{c@njecture}{\conjn@name}
\newcommand{\myl@bel}[2]{
 \protected@write \@auxout {}{\string \newlabel {#1}{{#2}{\thepage}{#2}{#1}{}} }
 \hypertarget{#1}{}
 } 
\newenvironment{labelledconj}[3][]
 {
 \def\conjn@name{#2}
 \begin{c@njecture}[{#1}]\myl@bel{#3}{#2}
 }
 {
 \end{c@njecture}
 }
\makeatother

\makeatletter
\newcommand{\mylabel}[2]{#2\def\@currentlabel{#2}\label{#1}}
\makeatother

\title[]{Mazur's Growth Number Conjecture in the Rank One Case}

\author[D.~Kundu]{Debanjana Kundu}
\address[Kundu]{Department of Mathematical and Statistical Sciences\\ UTRGV \\ 1201 W University Dr.\\ Edinburg, TX 78539\\ USA}
\email{debanjana.kundu@utrgv.edu}

\author[A.~Lei]{Antonio Lei}
\address[Lei]{Department of Mathematics and Statistics\\University of Ottawa\\
150 Louis-Pasteur Pvt\\
Ottawa, ON\\
Canada K1N 6N5}
\email{antonio.lei@uottawa.ca}

\keywords{Mazur's Growth Number Conjecture}
\subjclass[2020]{Primary 11R23, 11G05}

\begin{document}

\maketitle

\begin{abstract}
Let $p\geq 5$ be a prime number.
Let $\mathsf{E}/\mathbb{Q}$ be an elliptic curve with good ordinary reduction at $p$.
Let $K$ be an imaginary quadratic field where $p$ splits, and such that the generalized Heegner hypothesis holds. Under mild hypotheses, we show that if the $p$-adic height of the Heegner point of $\mathsf{E}$ over $K$ is non-zero, then Mazur's conjecture on the growth of Selmer coranks in the $\mathbb{Z}_p^2$-extension of $K$ holds.    
\end{abstract}

\section{Introduction}
Let $K$ be a number field. Let $\EC$ be an elliptic curve defined over $K$ and $L/K$ be an algebraic extension.
The Mordell--Weil theorem asserts that when $L/K$ is a \textit{finite} extension, the group $\EC(L)$ of $L$-rational points is a finitely generated $\Z$-module.
In other words, the (algebraic) rank of $\EC(L)$ is finite.
However, when $L$ is an {\it infinite} algebraic extension of $K$, the situation is more intricate.
For an elliptic curve $\EC/\Q$ and $p$ a prime number, the works of K.~Kato \cite{kato04} and D.~Rohrlich \cite{rohrlich88} imply that over the cyclotomic $\Zp$-extension $\Q_{\cyc}$ of $\Q$, the (algebraic) rank of $\EC(\Q_\cyc)$ is finite.
However, when $K/\Q$ is an imaginary quadratic field, and $K_\ac$ is the anti-cyclotomic $\Zp$-extension of $K$, it is possible that $\EC(K_\ac)$ has infinite rank; see \cite[(1.9) or (1.10)]{Gre_PCMS}.
In \cite[\S~18, p.~201]{Maz84}, B.~Mazur formulated a conjecture predicting that $K_\ac$ is the \textit{only} $\Zp$-extension of $K$ over which the Mordell--Weil rank can be infinite.

\begin{labelledconj}{Mazur's Growth Number Conjecture}{Mazur-conj} Suppose $\EC/\Q$ is an elliptic curve, $K$ is an imaginary quadratic field, and $p$ is a prime of good ordinary reduction for $\EC$.
Let $\cK/K$ be a $\Zp$-extension, and write $\cK_n$ to denote the unique subfield of degree $p^n$ of $K$ in $\cK$.
Further suppose that every prime of bad reduction of $\EC$ is finitely decomposed in $\cK/K$.
Then for $n \gg 0$
\[
\operatorname{corank}_{\Zp}\Sel_{p^\infty}(\EC/\cK_n) = cp^n + O(1),
\]
where 
\[
c = \begin{cases}
    0 & \textrm{ if } \cK \neq K_{\ac} \text{ or } (\EC,K) \textrm{ has `sign' } +1\\
    1 & \textrm{ if } \cK = K_{\ac}, (\EC,K) \textrm{ has `sign' } -1 \textrm{ and is `generic'}\\
    2 & \textrm{ if } \cK = K_{\ac}, (\EC,K) \textrm{ has `sign' } -1 \textrm{ and is `exceptional'}.
\end{cases}
\]
\end{labelledconj}

A few remarks are in order regarding the conjecture.
The integer $c$ is called the \emph{growth number}.
The pair $(\EC, K)$ is called `generic' if $\EC$ has no CM, or the CM field of $\EC$ is different from $K$.
On the other hand, the pair $(\EC, K)$ is called `exceptional' if $\EC$ has CM by (an order in) $K$.
The \emph{sign} of $(\EC, K)$ is defined in \cite[\S~6]{Maz84}) to mean
\begin{itemize}
    \item The sign of the functional equation of $L(\EC/K, s)$ in the `generic' case \emph{and}
    \item The sign of the functional equation of  $L(\varphi, s)$ in the `exceptional' case, where $\varphi$ is the Hecke character of $K$ satisfying
    \[
    L(\EC/\Q, s) = L(\varphi, s).
    \]
\end{itemize}
It is pertinent to note that the `sign' of $(\EC, K)$ is \textit{not} the same as the sign of the functional equation of $L(\EC/K, s)$ in the `exceptional' case.
 In particular, \ref{Mazur-conj} predicts that $c=0$ precisely
when $\cK \neq K_{\ac}$ or when one of the following situations is true
\begin{itemize}
    \item $(\EC, K)$ is `generic' and $L(\EC/K, s)$ has sign +1,
    \item $(\EC, K)$ is `exceptional' and $L(\EC/\Q, s) = L(\varphi, s)$ has sign +1.
\end{itemize}

\ref{Mazur-conj} is often formulated in terms of the growth of Mordell--Weil ranks.
The two formulations are equivalent under the condition that the $p$-parts of the Tate--Shafarevich groups are finite.
We assume that this is the case in the discussion below.

We record some instances in which the conjecture has been proven.
\begin{enumerate}
    \item[a)] The exceptional case of \ref{Mazur-conj} is well-understood from the works of K.~Rubin, D.~Rohrlich, B.~Gross--D.~Zagier, and R.~Greenberg (see for example \cite[Theorems~1.7 and 1.8]{Gre_PCMS}).
Let $\EC/\Q$ be an elliptic curve \textit{with complex multiplication} by an imaginary quadratic field $K$.
If $p$ is any prime, the Mordell--Weil rank of $\EC$ remains bounded in every $\Zp$-extension which is different from $K_{\ac}$.
Over the anti-cyclotomic $\Zp$-extension $K_{\ac}/K$, the Mordell--Weil rank of a CM elliptic curve is \emph{unbounded} if $p$ is a prime of good \textit{supersingular} reduction.
More explicitly, for $n\gg 0$
\[
\operatorname{rank}_{\Z} \EC(K_{\ac,n}) - \operatorname{rank}_{\Z} \EC(K_{\ac, n-1}) = \epsilon_n \phi(p^n),
\]
where $\phi(p^n)$ is the Euler totient function and $\epsilon_n = 0$ or 2, depending on the parity of $n$.
On the other hand, if $p$ is a prime of good \textit{ordinary} reduction then the boundedness of the Mordell--Weil rank of $\EC$ in $K_{\ac}/K$ depends on the parity of the order of vanishing of the Hasse--Weil $L$-function at $s=1$.
In particular, if $L(\EC/\Q, s)$ has an odd (resp. even) order 0 at $s=1$, then the rank growth is unbounded (resp. bounded).
In the case of unbounded rank growth, the rank increases by exactly $2\phi(p^n)$ at the $n$-th layer for $n\gg 0$.
\item[b)]More recently, there has been some partial progress made towards this conjecture for non-CM elliptic curves in \cite{GHKL, KMS}.
Let $K$ be an imaginary quadratic field in which $p$ is unramified and write $K_\infty$ for the $\Zp^2$-extension of $K$.
When $p$ is a prime of good \emph{ordinary} reduction, it was proven in \cite{KMS} that the number of $\Zp$-extensions of $K$ where the rank of $\EC$ does not stay bounded is at most $\min_H\{\lambda_H\}$, where $H$ runs over subgroups of $\Gal(K_\infty/K)$ such that $K_\infty^H$ is an admissible\footnote{in the sense of Definition~1.1 of \emph{op. cit.}} $\Zp$-extension of $K$ satisfying the $\mathfrak{M}_H(G)$-property, and $\lambda_H$ is the $\lambda$-invariant of the Pontryagin dual of the $p$-primary Selmer group of $E$ over that $\Zp$-extension.
In particular, the authors proved that \ref{Mazur-conj} holds when the Mordell--Weil rank of $\EC(K)$ is zero or one, under some hypotheses (see in particular Theorems 9.3 and 9.4 in \textit{op. cit.}).
On the other hand,  the authors in \cite{GHKL} proved Mazur's conjecture under a semi-cyclity hypothesis, which can be verified via computations, and is different from the hypotheses imposed in \cite{KMS}.
\item[c)] When $\EC/\Q$ has good \emph{supersingular} reduction at $p$, an analogue of Mazur's Conjecture was formulated in \cite{leisprung}.
For partial progress towards this conjecture in the case of non-CM elliptic curves, we refer the reader to \cite{leisprung,hunglim}.
For elliptic curves of rank $0$ over $K$ and $p\geq 3$ a prime of supersingular reduction with $a_p =0$ such that both primes above $p$ in $K$ are completely ramified in $K_{\ac}/K$, this conjecture was resolved in \cite{GHKL}.
\end{enumerate}

The main goal of this article is to provide an \textit{analytic proof} of \ref{Mazur-conj} for elliptic curves that are of rank $1$ over $K$ when $p$ is a prime of good ordinary reduction.
More precisely, we prove the following theorem.

\begin{lthm}\label{thmA}
Fix a prime $p\geq 5$.
Let $\EC/\Q$ be a non-CM elliptic curve with good ordinary reduction at $p$.
Let $K$ be an imaginary quadratic field where $p$ splits, and such that \eqref{GHH}, \eqref{h-IMC} and \eqref{tor} hold. Writing $z_{\mathrm{Heeg}}$ to denote the Heegner point of $\EC$ over $K$, if the $p$-adic height $ \langle z_{\mathrm{Heeg}}, z_{\mathrm{Heeg}} \rangle_0\ne 0$, then \ref{Mazur-conj} holds.    
\end{lthm}

We refer the reader to \cite[p.~29]{zhang2001heights} for construction of the Heegner point.
Although the original formulation of \ref{Mazur-conj} assumes that every
prime of bad reduction of $\EC$ is finitely decomposed in $K_\infty/K$, this condition is not required in our proof of Theorem~\ref{thmA}. The generalized Heegner hypothesis \eqref{GHH} requires that the number of bad primes that are inert in $K$ be even.
In particular, these primes (if they exist) split completely in $K_\ac$, and therefore are not finitely decomposed in $K_\infty$.
To prove our result in full generality, we require one inclusion of the Iwasawa main conjecture over $K_\infty/K$; a proof of this inclusion has recently been announced under some mild hypotheses in \cite{BCS24,BSTW}.

Our proof of Theorem~\ref{thmA} is analytic in the sense that it uses the theory of $p$-adic $L$-functions.
It is based on the work of D.~Disegni \cite{disegni17}, which provides a formula relating the cyclotomic derivative of a 2-variable $p$-adic $L$-function to the cyclotomic $p$-adic
heights of Heegner points over the anti-cyclotomic tower. Note that such formula was previously proved by F.~Castella \cite{Cas17} and B.~Howard \cite{How05} under additional hypotheses.

\emph{Outlook:}
The algebraic proof of \ref{Mazur-conj} in \cite{GHKL} could not be extended to the case where $\EC/\Q$ is an elliptic curve without complex multiplication that has positive rank over $K$ and $p$ is a prime of good supersingular reduction.
The analytic proof of Mazur's conjecture presented in this article does not generalize immediately either (see Remark~\ref{rk:ss}).
New ideas seem to be necessary for further study of the supersingular case.

The known algebraic and analytic methods for proving Mazur's conjecture for non-CM elliptic curves crucially require that the Selmer corank of $\EC/K$ be at most 1.
An obvious direction of future investigation is to better understand the higher corank situation.

Finally, we expect that our proof of Theorem~\ref{thmA} can be extended to abelian varieties of $\GL_2$-type that are defined over $\Q$. 
One may even consider abelian varieties attached to Hilbert modular forms and study how the Selmer corank grows over abelian $p$-adic extensions of CM fields since the work of Disegni \cite{disegni17,disegni} applies in this generality.

After the first version of the article was written, the preprint \cite{RM-preprint} appeared on arXiv. In this work, similar ideas to those utilized in this article were used to study Hilbert's 10th problem in $\Zp$-towers of number fields over $K$.
The results therein depend on the semi-cyclicity hypothesis introduced in \cite{GHKL}.

\section*{Acknowledgement}
This work was started during AL's visit to UTRGV in spring 2024.
He thanks the School of Mathematical and Statistical Sciences at UTRGV for their hospitality.
The authors thank Francesc Castella, Jeff Hatley, and Rylan Gajek-Leonard for interesting discussions during the preparation of this article.
We also thank Daniel Disegni for answering our questions related to his work. 
AL's research is supported by the NSERC Discovery Grants Program RGPIN-2020-04259.
Finally, we thank the referee for their helpful feedback and comments on an earlier version of the article, which led to many improvements.

\section{Notation and Preliminaries}

\subsection{Elliptic curves and Selmer groups}
We collect the notation and assumptions that will be in place throughout the article.

Let $\EC/\Q$ be a non-CM elliptic curve of conductor $N=N_{\EC}$ with good ordinary reduction at a prime $p\geq 5$.
Let $K$ be an imaginary quadratic field with $p\cO_K = \fp \overline{\fp}$, where $\fp\ne\overline{\fp}$.
We write $K_\infty/K$ to denote the $\Zp^2$-extension of $K$ and set $G_\infty=\Gal(K_\infty/K)$.
The Iwasawa algebra $\Zp\llbracket G_\infty\rrbracket$ is denoted by $\Lambda$.

We write $N = N^+ N^-$ with $N^+$ (respectively $N^-$) equal to the product of the prime factors of $N$ which are split (respectively inert) in $K$.
Throughout this article, we assume that the pair $(\EC,K)$ satisfies the \textbf{generalized Heegner hypothesis}.
More concretely,\vspace{0.2cm}
\begin{itemize}
\item[(\mylabel{GHH}{\textbf{GHH}})]  $N^-$ is the square-free product of an \textit{even} number of primes.
\end{itemize}
\vspace{0.2cm}

Let $\Sigma$ be a finite set of primes in $\Q$ containing $p$ and all primes of bad reduction for $\EC$.
For any field $F/\Q$, define $\Sigma(F)$ to be the set of places of $F$ lying above those in $\Sigma$, and write $G_\Sigma(F)$ for the Galois group of the maximal extension of $F$ that is unramified outside of $\Sigma(F)$.
Furthermore, for any $v \in \Sigma$ and any finite extension $F/\Q$, write 
\[
J_v(\EC/F)= \bigoplus_{w\mid v} H^1(F_w,\EC)[p^\infty].
\]
When $\mathcal{F}/F$ is an infinite extension of $F$, set
\[
J_v(\EC/\mathcal{F}) = \varinjlim_{F\subseteq F' \subseteq \mathcal{F}} J_v(\EC/F').
\]

\begin{definition}\label{def:selmer-ord} Let $\EC/\Q$ be an elliptic curve.
Let $\Sigma$ be any finite set of primes containing those dividing $pN$.
For any extension $L/\Q$, define the Selmer group 
\[
\Sel_{p^\infty}(\EC/L):=\ker \left( H^1(G_{\Sigma}(L),\EC[p^\infty]) \rightarrow 
\prod_{v \in \Sigma} J_v(\EC/L) \right).
\]
\end{definition}

\subsection{\texorpdfstring{$\Zp$}{}-extensions}\label{sec:zp-extensions}
Let $K_\cyc$ and $K_\ac$ denote the cyclotomic and anti-cyclotomic $\Zp$-extension of $K$, respectively.
We fix two topological generators $\sigma$ and $\tau$ of $G_\infty$ such that
\begin{align*}
\overline{\langle \sigma\rangle}&=\ker\left(G_\infty\longrightarrow\Gal(K_\cyc/K)\right),\\
\overline{\langle \tau\rangle}&=\ker\left(G_\infty\longrightarrow\Gal(K_\ac/K)\right).
\end{align*}
We may identify $\Lambda$ with the ring of power series $\Zp\llbracket X,Y\rrbracket$ by sending $\sigma$ to $X+1$ and $\tau$ to $Y+1$.

\begin{lemma}
\label{lem:Zp-ext}
Let $\mathcal{K}/K$ be a $\Zp$-extension.
There exists a unique element $(a:b)\in\PP^1(\Zp)$ such that 
\[
\overline{\langle \sigma^a\tau^b\rangle}=\ker\left(G_\infty\rightarrow\Gal(\mathcal{K}/K)\right),
\]
\end{lemma}

\begin{proof} 
Each $\Zp$-extension corresponds to a unique subgroup of $H$ of $G_\infty$ such that $G_\infty/H\simeq\Zp$.
Since $G_\infty=\overline{\langle\sigma,\tau\rangle}\simeq\Zp^2$, the lemma follows.
\end{proof}

Given $(a:b)\in\PP^1(\Zp)$, we shall always pick a representative where $a$ and $b$ are elements of $\Zp$ such that $\min(\ord_p(a),\ord_p(b))=0$.

\begin{notn}
Let $\mathcal{K}/K$ be a $\Zp$-extension.
We adopt the following notation throughout the article.
\begin{itemize}
\item Set $\mathcal{K}=K_{a,b}$, where $(a:b)\in \PP^1(\Zp)$ is given as in Lemma~\ref{lem:Zp-ext}.
\item Set $\Gamma_{a,b}=\Gal(K_{a,b}/K)$, $H_{a,b}=\Gal(K_\infty/K_{a,b})$, and $\Lambda_{a,b}=\Zp\llbracket\Gamma_{a,b}\rrbracket$.
\item Set $\pi_{a,b}:\Lambda\rightarrow \Lambda_{a,b}$ for the map induced by the natural projection map $G_\infty\rightarrow\Gamma_{a,b}$.
\item Set $f_{a,b}=(1+X)^a(1+Y)^b-1$.
\item Set $\Gamma_\cyc$ and $\Lambda_{\cyc}$ (respectively, $\Gamma_{\ac}$ and $\Lambda_{\ac}$) as $\Gamma_{1,0}$ and $\Lambda_{1,0}$ (respectively, $\Gamma_{0,1}$ and $\Lambda_{0,1}$). 
\end{itemize}
\end{notn}
Note that $f_{a,b}$ is an irreducible element of $\Lambda$ and we have the identification $\Lambda_{a,b}=\Lambda/(f_{a,b})$. When $(a:b)=(1:0)$, this gives $\Lambda_\cyc= \Zp\llbracket Y\rrbracket$, whereas $(a:b)=(0:1)$ gives $\Lambda_\ac= \Zp\llbracket X\rrbracket$. Furthermore, if $(a:b)\neq (a':b')$, then $f_{a,b}$ and $f_{a',b'}$ are coprime to each other as elements of $\Lambda$.

\subsection{Pseudo-null Iwasawa modules}

We review the definition and basic properties of pseudo-null modules over the Iwasawa algebra $\Lambda$.
\begin{definition}
Let $M$ be a finitely generated $\Lambda$-module.
Let $\fP$ be a prime ideal of $\Lambda$.
\item[\textup{i)}] The localization of $M$ at $\fP$ is denoted by $M_\fP$.
\item[\textup{ii)}] A $\Lambda$-module $M$ is called \textbf{pseudo-null} if $M_\fP=0$ for all prime ideals of $\Lambda$ of height $\le 1$.
Equivalently, if $\fP$ is a prime ideal such that $\ann_\Lambda(M)\subseteq \fP$, then the height of $\fP$ is at least 2.
\end{definition}

Let $M$ be a $\Lambda$-module and $(a:b)\in \PP^1(\Zp)$.
As $\Gamma_{a,b}=G_\infty/H_{a,b}$, the action of $\Lambda$ on $M$ induces a natural $\Lambda_{a,b}$-module structure on the homology groups $H_i(H_{a,b}, M)$ for $i\ge0$.
When $M$ is a finitely generated pseudo-null $\Lambda$-module, we have:

\begin{lemma}
\label{cohomology groups of pseudo-null modules are torsion}
Let $M$ be a finitely generated pseudo-null $\Lambda$-module.
Let $(a:b)\in \PP^1(\Zp)$.
The homology groups $H_0(H_{a,b}, M)$ and $H_1(H_{a,b}, M)$ are $\Lambda_{a,b}$-torsion. 
\end{lemma}

\begin{proof}
Since $M$ is a $\Lambda$-pseudo-null module, it has Krull dimension at most 1. 
The $\Lambda_{a,b}$-modules $H_1(H_{a,b},M)$ and $H_0(H_{a,b}, M)$ also have Krull dimension  $\leq 1$.
In particular, they are $\Lambda_{a,b}$-torsion. 
\end{proof}

\subsection{\texorpdfstring{$p$-adic $L$-function}{}}

Since $p\geq 5$ is assumed to be a prime of good ordinary reduction, by the work of H.~Hida \cite{Hid85} and B.~Perrin-Riou \cite{PR87-Invent, PR88} (see also \cite[Theorem~A]{disegni17}), there exists a 2-variable $p$-adic $L$-function $L_p(X,Y) \in \Q_p \otimes_{\Zp} \Lambda$ interpolating the central critical values for the Rankin--Selberg convolution of (the modular form attached to) $\EC$ with the theta series attached to finite order characters of $G_\infty$.
We write $L_p(X,0)$ to denote its projection to the anti-cyclotomic line.
More precisely, it is the image under the map $\Qp \otimes_{\Zp} \Lambda \rightarrow \Qp \otimes_{\Zp} \Lambda_{\ac} $ induced by the natural projection $G_\infty \rightarrow \Gamma_{\ac}$.

Under the identification of $\Lambda$ with $\Lambda_{\ac}\llbracket \Gamma_{\cyc}\rrbracket=\Zp\llbracket X\rrbracket\llbracket Y\rrbracket$, we can expand
\begin{equation}
\label{power series in Y}
L_p(X,Y) = L_p(X,0) + \frac{\partial L_p}{\partial Y}(X,0) Y +  \frac{1}{2}\cdot\frac{\partial^2 L_p}{\partial Y^2}(X,0) Y^2 + \ldots
\end{equation}
as a power series in $Y= \tau - 1$ with coefficients in $\Qp\otimes_{\Zp}\Zp\llbracket X\rrbracket$.
Moreover, in view of the assumption \eqref{GHH}, the sign of the functional equation of $L(E/K,s)$ is $-1$, which forces $L_p(X,0)=0$; see \cite[Proposition~1.4]{Cas17}.

\subsubsection{\texorpdfstring{$\Lambda_{\ac}$-adic height pairing}{}}
\label{sec: Lambda adic height pairing}
Let $T=\displaystyle \varprojlim_{m}\EC[p^m]$ denote the $p$-adic Tate module of $\EC$. 
Since we assume that \eqref{GHH} holds, we can view $\EC$ as a quotient of the Jacobian of a Shimura curve attached to the pair $(N^+, N^-)$.
Taking linear combinations of Heegner points on the Shimura curve $X_{N^+, N^-}$ defined over ring class fields of $K$ of $p$-power conductor, and mapping them onto $\EC$ via the fixed parameterization (possibly after replacing $\EC$ by an isogenous elliptic curve)
\[
\phi(\EC): \mathrm{Jac}(X_{N^+, N^-}) \longrightarrow \EC,
\]
we can define a system of Heegner classes $z_\infty=(z_n)_{n\ge0}$ in the \textit{compact} Selmer group
\[
\Sel(K, T^{\ac}) := \varprojlim_{n} \varprojlim_{m} \Sel_{p^m}(\EC/K_{\ac,n})
\]
attached to the elliptic curve $\EC$ over $K_{\ac}$.
Here, $\displaystyle z_n\in \varprojlim_m\Sel_{p^m}(\EC/ K_{\ac,n})$, and  $\Sel_{p^m}(\EC/ K_{\ac,n}) \subseteq H^1(K_{\ac,n},\EC[p^m])$ is the usual $p^m$-torsion Selmer group over the $n$-th layer of the anti-cyclotomic $\Zp$-extension of $K$. Furthermore, $z_0$ corresponds to the image of the Heegner point $z_\mathrm{Heeg}$ in the statement of Theorem~\ref{thmA} under the Kummer map.

Write $\mathcal{I}$ to denote the augmentation ideal of $\Lambda_{\cyc}$.
By the work of Perrin-Riou \cite{PR92}, for every $n\geq 0$, there is a $p$-adic height pairing
\[
\langle- ,- \rangle_{n} : \Sel(K_{\ac,n}, T) \times \Sel(K_{\ac,n}, T) \longrightarrow p^{-k}\Zp \otimes_{\Zp} \mathcal{I}/\mathcal{I}^2,
\]
for some non-negative integer $k$, which does not depend on $n$.
Next, define the (cyclotomic) $\Lambda_{\ac}$-adic height pairing
\[
\langle- ,- \rangle_{\Lambda_{\ac}} : \Sel(K, T^{\ac}) \times \Sel(K, T^{\ac}) \longrightarrow \Q_p \otimes_{\Zp}\Lambda_{\ac} \otimes_{\Zp} \mathcal{I}/\mathcal{I}^2
\]
by the formula
\[
\langle a_\infty , b_\infty \rangle_{\Lambda_{\ac}} = \varprojlim_{n} \sum_{\theta \in \Gal(K_{\ac,n}/K)} \langle a_n , b_n^\theta \rangle_{n}\cdot \theta
\]
(see also \cite[\S~11]{nekovar06}).
Using the fixed topological generator $\tau$, it is possible to view $\langle - , - \rangle_{\Lambda_{\ac}}$ as taking values in $\Q_p \otimes_{\Zp}\Lambda_{\ac}$.

We now record a result of Disegni, which we will require crucially in our proof of Theorem~\ref{thmA}.

\begin{theorem}
\label{thm:Castella}
Fix an odd prime $p\geq 5$.
Let $\EC/\Q$ be an elliptic curve with good ordinary reduction at $p$.
Let $K$ be an imaginary quadratic field where $p$ splits and such that the pair $(\EC,K)$ satisfies \eqref{GHH}.
Then the following equality holds (as ideals of $\Q_p \otimes_{\Zp} \Lambda_{\ac}$)
\[
\Big\langle \frac{\partial L_p}{\partial Y}(X,0) \Big\rangle = \Big\langle \langle z_\infty, z_\infty \rangle_{\Lambda_{\ac}} \Big\rangle.
\]
\end{theorem}

\begin{proof}
This is a special case of \cite[Theorem~C(4)]{disegni17}; the totally real field $F$ (resp. the CM field $E$) is set as $\Q$ (resp. $K$), and $A$ and $A^\vee$ in \textit{loc. cit.} are chosen as $\EC$. See also \cite[Theorem~4.3.3]{disegni20}, where this result is restated; the appropriate choice of $f\in\pi_A$ is discussed in the proof therein, relying on results in \cite{CST}.
\end{proof}

If $N^-=1$, Theorem~\ref{thm:Castella} was established in \cite[Theorem~B]{How05} when $K\neq \Q(\sqrt{-3})$ and the discriminant of $K$ is odd. A similar result was proved in 
\cite[Theorem~A]{Cas17} when $N$ is square-free, $\EC[p]$ is ramified at every prime $\ell \mid N^-$, and the representation $\Gal(\overline{\Q}/K) \rightarrow \mathrm{Aut}_{\Zp}(T)$ is surjective.

\subsection{Iwasawa Main Conjecture}
\label{sec: IMC}
The main result of this article requires the following hypothesis, i.e., one inclusion of the Iwasawa main conjecture
\begin{equation}
    \label{h-IMC} \tag{\textbf{h-IMC}} L_p(X,Y) \in \Char_{\Lambda}(\Sel_{p^\infty}(\EC/K_\infty)^\vee).
\end{equation}
In particular, since the characteristic ideal of $\Sel_{p^\infty}(\EC/K_\infty)^\vee$ lies inside $\Lambda$ by definition, it is assumed implicitly that $L_p(X,Y)\in\Lambda $ (rather than $\Qp\otimes_{\Zp}\Lambda$). For notational simplicity, we shall let $I$ denote the characteristic ideal $\Char_{\Lambda}(\Sel_{p^\infty}(\EC/K_\infty)^\vee)$.

For $(a:b)\in\PP^1(\Zp)$, let $\cL_{a,b}=\pi_{a,b}\left(L_p(X,Y)\right)$. Under \eqref{h-IMC}, taking projection $\pi_{a,b}$ gives
\[
\cL_{a,b}  \in\pi_{a,b}(I).
\]
In particular, it tells us that if $f_{a,b}\nmid L_p(X,Y)$ over $\Lambda$, then $\pi_{a,b}(I)\ne\{0\}$.

The Iwasawa main conjecture for $\EC$ over $K_\infty$ predicts that the following equality of $\Lambda$-ideals holds.
\begin{equation}
\label{IMC}
\tag{\textbf{IMC}} \langle L_p(X,Y) \rangle = I.     
\end{equation}

A proof of \eqref{IMC} was recently announced by A.~Burungale, F.~Castella, and C.~Skinner in \cite{BCS24} for imaginary quadratic fields $K\neq \Q(\sqrt{-3})$ with odd discriminant, under some mild technical hypotheses.
In particular, they work under the assumption that $N^- = 1$, i.e., the strict Heegner hypothesis holds.
Moreover, they assume that $\EC[p]$ is an irreducible $\Gal(\overline{\Q}/\Q)$-module and that the set of `vexing primes' is empty.\footnote{Further progress on \eqref{IMC} has been made in \cite{BSTW} after the first version of this article was written.}

\begin{remark}
Up to normalizing factors (involving congruence number, modular degree, and Manin constant), the $p$-adic $L$-function $\cL_p^{\mathrm{PR}}$ (in \cite{BCS24}) is the same as the 2-variable $p$-adic Rankin $L$-series $L_p(f/K, \Sigma^{(1)})$ attached to a modular form $f$ (in \cite{Cas17,disegni17}) which we denote by $L_p(X,Y)$ in this article.
\end{remark}

\section{Mazur's growth number and the cyclotomic Selmer group}
\label{sec: Cyclotomic}

\subsection{\texorpdfstring{$p$-adic $L$-function over $K_{a,b}$}{}}
\label{sec: p-adic L function over Kab}

In this section, the goal is to study the non-vanishing of the $p$-adic $L$-function $\cL_{a,b}=\pi_{a,b}(L_p(X,Y))$ over the $\Zp$-extension $K_{a,b}$ for when $K_{a,b}\neq K_{\ac}$.

\begin{theorem}\label{thm:derivative}
Let $(a:b)\in\PP^1(\Zp)\setminus\{ (0:1)\}$.
Assume that the hypotheses of Theorem~\ref{thm:Castella} hold and that the $p$-adic height $\langle z_{\mathrm{Heeg}}, z_{\mathrm{Heeg}} \rangle_0$ is non-zero.
Then $\cL_{a,b}\ne0$.
\end{theorem}


\begin{proof}
Recall from \eqref{power series in Y} that
\[
L_p(X,Y) = \frac{\partial L_p}{\partial Y}(X,0) Y + \frac{1}{2}\cdot \frac{\partial^2 L_p}{\partial Y^2}(X,0) Y^2 + \ldots
\]
Let $(a:b)\in \PP^1(\Zp)$ such that $a \neq 0$.
After multiplying by a non-zero scalar if necessary, we may assume that $a\in \Zp^\times$ or $b\in\Zp^\times$.
We consider the two cases separately. 
\subsubsection*{\underline{Case 1: \texorpdfstring{$b\in\Zp^\times$}{}}}

It follows from Lemma~\ref{lem:Zp-ext} that
\[
\Gamma_{a,b} = \frac{\overline{\langle \sigma, \tau\rangle}}{\overline{\langle \sigma^a \tau^b\rangle}} = \frac{\overline{\langle \sigma, \tau\rangle}}{\overline{\langle \sigma^{\frac{a}{b}} \tau\rangle}} = \overline{\langle \pi_{a,b}(\sigma)\rangle}
\]
since $\pi_{a,b}(\tau)=\pi_{a,b}(\sigma^{-\frac{a}{b}})$. In particular, the Iwasawa algebra $\Lambda_{a,b}$ may be identified as the power series ring $\Zp\llbracket \pi_{a,b}(X)\rrbracket$. Furthermore, 
\[
Y \equiv (1+X)^{-\frac{a}{b}} - 1 \pmod{f_{a,b}}.
\]
 Therefore, we have
\[
L_p(X,Y) \equiv  \frac{\partial L_p}{\partial Y}(X,0) [(1+X)^{-\frac{a}{b}} - 1] +  \frac{1}{2}\cdot\frac{\partial^2 L_p}{\partial Y^2}(X,0) [(1+X)^{-\frac{a}{b}} - 1]^2 + \ldots\pmod{f_{a,b}}
\]
Set $X_{a,b}=\pi_{a,b}(X)$ and $\Lambda_{a,b}=\Zp\llbracket X_{a,b}\rrbracket$. Identifying $\cL_{a,b}$ with $L_p(X,Y)\pmod{f_{a,b}}$,
we have the following congruences of elements in $\Lambda_{a,b}$
\[
\cL_{a,b} \equiv - \frac{a}{b}\cdot\frac{\partial L_p}{\partial Y}(X_{a,b},0)\cdot X_{a,b} \pmod{\langle X_{a,b}^2\rangle}, 
\]
since $(1+X)^{-\frac{a}{b}}\equiv 1-\frac{a}{b}\cdot X\pmod{\langle X^2\rangle}$.
Taking derivative with respect to $X_{a,b}$ gives
\[
\frac{\partial \cL_{a,b}}{\partial X_{a,b}}\Bigg\vert_{X_{a,b}=0} = -\frac{a}{b}\cdot \frac{\partial L_p}{\partial Y}(0,0).
\]
\vspace{0.2cm}
\subsubsection*{\underline{Case 2: \texorpdfstring{$a\in\Zp^\times$}{}}}
Similar to the previous case, we have $\Lambda_{a,b}=\Zp\llbracket Y_{a,b}\rrbracket$, where $Y_{a,b}=\pi_{a,b}(Y)$ and 
\[
X \equiv (1+Y)^{-\frac{b}{a}} - 1 \pmod{f_{a,b}}.
\]
We deduce analogously
\[
L_p(X,Y) \equiv  \frac{\partial L_p}{\partial Y}((1+Y)^{-\frac{b}{a}} - 1,0) Y +  \frac{1}{2}\cdot\frac{\partial^2 L_p}{\partial Y^2}((1+Y)^{-\frac{b}{a}} - 1,0)Y^2 + \ldots\pmod{f_{a,b}}.
\]
Therefore, the following congruence of elements of $\Lambda_{a,b}$ holds
\[
\cL_{a,b} \equiv \frac{\partial L_p}{\partial Y}((1+Y_{a,b})^{-\frac{b}{a}}-1,0)Y_{a,b} \pmod{\langle Y_{a,b}^2 \rangle }.
\]
Taking derivative with respect to $Y_{a,b}$, we obtain the formula
\[
\frac{\partial \cL_{a,b}}{\partial Y_{a,b}}\Bigg\vert_{Y_{a,b}=0}=\frac{\partial L_p}{\partial Y}((1+0)^{-\frac{b}{a}}-1,0) = \frac{\partial L_p}{\partial Y}(0,0).
\]
\vspace{0.4cm}

In both cases, we can write $\Lambda_{a,b}=\Zp\llbracket Z_{a,b}\rrbracket$, where $Z_{a,b}$ is either $X_{a,b}$ or $Y_{a,b}$, and we have shown that the derivative of $\cL_{a,b}$ with respect to $Z_{a,b}$ evaluated at $Z_{a,b}=0$ is equal to a non-zero multiple of $\frac{\partial L_p}{\partial Y}(0,0)$ (since we exclude the case where $a=0$).
Recall from Theorem~\ref{thm:Castella} that the following equality of $\Qp\otimes_{\Zp}\Zp\llbracket X\rrbracket$-ideals holds
\[
\Big\langle \frac{\partial L_p}{\partial Y}(X,0) \Big\rangle = \Big\langle  \langle z_\infty, z_\infty \rangle_{\Lambda_{\ac}} \Big\rangle.
\]
In particular, we conclude from the discussion in \S~\ref{sec: Lambda adic height pairing} on the height pairing that  $\frac{\partial L_p}{\partial Y}(0,0)$ is equal to  a non-zero multiple of $ \langle z_{\mathrm{Heeg}}, z_{\mathrm{Heeg}} \rangle_0$. In particular, $\cL_{a,b}\ne 0$ whenever $ \langle z_{\mathrm{Heeg}}, z_{\mathrm{Heeg}} \rangle_0\ne0$.
\end{proof}

The $p$-adic height $ \langle z_{\mathrm{Heeg}}, z_{\mathrm{Heeg}} \rangle_0$ is provably non-zero when $\EC$ is a CM elliptic curve and $p$ is a prime of ordinary reduction \cite{Ber83}.
At primes of supersingular reduction of $\EC$ (both in the CM and non-CM case), similar results are proven in \cite{BPR93, kobayashi2013p}.
When $\EC$ is Eisenstein at $p$, results of \cite{BS23} give sufficient conditions (in terms of the $\lambda$-invariant of a certain Selmer group) for the $p$-adic height to be non-zero.
In general, the $p$-adic height of Heegner points can be computed explicitly using the algorithm in \cite{BaCiSt15}, and is expected to be non-zero when $\EC(K)$ is of rank one since the $p$-adic Birch and Swinnerton-Dyer conjecture (as formulated in \cite[\S~10]{MTT}) predicts that the derivative of the cyclotomic $p$-adic $L$-function $\cL_{1,0}$ should be non-zero.

\subsection{Boundedness of Selmer coranks}
We show that the dual Selmer group $\Sel_{p^\infty}(\EC/K_{a,b})^\vee$ is $\Lambda_{a,b}$-torsion whenever $(a:b)\notin\{(0:1)\}$.
For the remainder of the discussion we assume the following hypothesis:
\vspace{0.2cm}
\begin{itemize}
    \item[(\mylabel{tor}{\textbf{tor}})]  $\EC(K)$ has no $p$-torsion.
\end{itemize}
\vspace{0.2cm}

\begin{lemma}
\label{no non trivial pseudo-null}
If \eqref{tor} holds, then $\Sel_{p^\infty}(\EC/K_\infty)^\vee$ contains no non-trivial pseudo-null $\Lambda$-submodule.    
\end{lemma}

\begin{proof}
See \cite[Proposition~6.1]{KLR}.    
\end{proof}

By the results of K.~Kato \cite{kato04} and D.~Rohrlich \cite{rohrlich88}, we know that $\Sel_{p^\infty}(\EC/K_\cyc)^\vee$ is $\Lambda_{\cyc}$-torsion since $\EC$ is an elliptic curve defined over $\Q$ and $K$ is an imaginary quadratic number field (and hence an abelian extension of $\Q$).
It implies that $\Sel_{p^\infty}(\EC/K_\infty)^\vee$ is $\Lambda$-torsion; see  \cite[Lemma~2.6]{HO10}, whose proof is based on results in \cite{BH97}).
The structure theorem of finitely generated $\Lambda$-modules combined with Lemma~\ref{no non trivial pseudo-null} asserts the existence of the following short exact sequence
\begin{equation}
\label{eq:structure}    
0 \longrightarrow \Sel_{p^\infty}(\EC/K_\infty)^\vee \longrightarrow \bigoplus_{i=1}^m\frac{\Lambda}{I_i} \longrightarrow N \longrightarrow 0, 
\end{equation}
where $I_1,\dots,I_m $ are principal ideals of $\Lambda$ and $\prod_{i=1}^mI_i=I = \Char_{\Lambda}(\Sel_{p^\infty}(\EC/K_\infty)^\vee)$ is the characteristic ideal of $\Sel_{p^\infty}(\EC/K_\infty)^\vee$ as a $\Lambda$-module and $N$ is a pseudo-null $\Lambda$-module.

\begin{lemma}
\label{SHab is cotorsion}
Let $(a:b)\in\PP^1(\Zp)\setminus\{(0:1)\}$.
Suppose that \eqref{GHH}, \eqref{tor}, and \eqref{h-IMC} hold.
If $\langle z_{\mathrm{Heeg}}, z_{\mathrm{Heeg}} \rangle_0$ is non-zero, then $\Sel_{p^\infty}(\EC/K_\infty)_{H_{a,b}}^\vee$ is a finitely generated torsion $\Lambda_{a,b}$-module.
\end{lemma}

\begin{proof}
To prove the lemma, we consider the following exact sequence
\[
H_1(H_{a,b}, N) \longrightarrow \Sel_{p^\infty}(\EC/K_\infty)^{\vee}_{H_{a,b}} \longrightarrow\bigoplus_{i=1}^m \frac{\Lambda_{a,b}}{\pi_{a,b}(I_i)} \longrightarrow H_0(H_{a,b}, N) \longrightarrow 0,
\]
which is induced by the short exact sequence \eqref{eq:structure}.
Theorem~\ref{thm:derivative} says that $\cL_{a,b}\ne0$ under our running hypotheses. Thus, \eqref{h-IMC} implies that $\pi_{a,b}(I)\ne0$. In particular, $\frac{\Lambda_{a,b}}{\pi_{a,b}(I_i)}$ is $\Lambda_{a,b}$-torsion for $i=1,\dots,m$.
Since $N$ is a pseudo-null $\Lambda$-module, Lemma~\ref{cohomology groups of pseudo-null modules are torsion} asserts that $H_0(H_{a,b}, N)$ and $H_1(H_{a,b}, N)$ are both $\Lambda_{a,b}$-torsion.
Hence, $\Sel_{p^\infty}(\EC/K_\infty)_{H_{a,b}}^\vee$ is a $\Lambda_{a,b}$-torsion module.
\end{proof}

\begin{proposition}
Under the same assumptions as Lemma~\ref{SHab is cotorsion}, the Selmer group $\Sel_{p^\infty}(\EC/K_{a,b})^\vee$ is a finitely generated $\Lambda_{a,b}$-torsion module.    \label{prop:Sel-tor}
\end{proposition}

\begin{proof}
Recall the following fundamental diagram
\[
\begin{tikzcd}
0 \arrow[r] &
\Sel_{p^\infty}(\EC/K_{a,b}) \arrow[r] \arrow[d,"\alpha"] &
H^1(G_{\Sigma}(K_{a,b}),\EC[p^\infty]) \arrow[r,"\pi"] \arrow[d,"\beta"] &
\prod_{v \in \Sigma(\cK)} J_v(\EC/K_{a,b}) \arrow[d,"\gamma=
\prod \gamma_v"] 
\\
0 \arrow[r] &
\Sel_{p^\infty}(\EC/K_\infty)^{H_{a,b}} \arrow[r] &
H^1(G_{\Sigma}(K_\infty),\EC[p^\infty])^{H_{a,b}} \arrow[r] &
\prod_{v \in \Sigma(K_\infty)} J_v(\EC/K_\infty)^{H_{a,b}}, 
\end{tikzcd}
\]
which is induced from the central map $\beta$ given by restriction in cohomology.
The first column of the diagram gives the exact sequence
\begin{equation}\label{eq:kernel}
0 \longrightarrow \ker{\alpha} \longrightarrow \Sel_{p^\infty}(\EC/K_{a,b}) \longrightarrow \Sel_{p^\infty}(\EC/K_\infty)^{H_{a,b}}.
\end{equation}
By the inflation-restriction exact sequence,
\[
\ker \beta = H^1(H_{a,b}, \EC(K_\infty)[p^\infty]).
\]
Thus, it follows from the commutative diagram above that
\[
\ker \alpha \hookrightarrow H^1(H_{a,b}, \EC(K_\infty)[p^\infty]).
\]
Since we assume that $\EC(K)[p]=0$, it follows that $\EC(K_\infty)[p]=0$ (see for example, \cite[I.6.13]{NSW}); hence $\ker(\alpha)=0$.
Therefore, taking duals in \eqref{eq:kernel} gives a surjection
\[
\begin{tikzcd} 
\Sel_{p^\infty}(\EC/K_\infty)_{H_{a,b}}^\vee \arrow[r, tail, twoheadrightarrow] & \Sel_{p^\infty}(\EC/K_{a,b})^\vee.
\end{tikzcd}
\]
We have shown in Lemma~\ref{SHab is cotorsion} that $\Sel_{p^\infty}(\EC/K_\infty)_{H_{a,b}}^\vee$ is a finitely generated $\Lambda_{a,b}$-torsion module.
Hence, we conclude  $\Sel_{p^\infty}(\EC/K_{a,b})^\vee$ is a finitely generated torsion $\Lambda_{a,b}$-module.
\end{proof}

We are now ready to give the proof of Theorem~\ref{thmA}.

\begin{theorem*}
Fix an odd prime $p\geq 5$.
Let $\EC/\Q$ be a non-CM elliptic curve with good ordinary reduction at $p$.
Let $K$ be an imaginary quadratic field where $p$ splits, and such that \eqref{GHH}, \eqref{h-IMC} and \eqref{tor} hold. If the $p$-adic height $ \langle z_{\mathrm{Heeg}}, z_{\mathrm{Heeg}} \rangle_0\ne 0$, then \ref{Mazur-conj} holds.    
\end{theorem*}

\begin{proof}
Let $(a:b)\in\PP^1(\Zp)\setminus\{(0:1)\}$.
It follows from Proposition~\ref{prop:Sel-tor} that the $\Lambda_{a,b}$-module $\Sel_{p^\infty}(\EC/K_{a,b})^\vee$ is finitely generated and torsion.
Let $L$ be a finite subextension of $K_{a,b}/K$.
Then hypothesis \eqref{tor} implies that the restriction map
\[
H^1(L,\EC[p^\infty])\longrightarrow H^1(K_{a,b},\EC[p^\infty])
\]
is injective. In particular, this induces an injection
\[
\Sel_{p^\infty}(\EC/L)\hookrightarrow \Sel_{p^\infty}(\EC/K_{a,b}).
\]
Therefore, the $\Zp$-corank of $\Sel_{p^\infty}(\EC/L)$ is uniformly bounded by the $\Zp$-rank of $\Sel_{p^\infty}(\EC/K_{a,b})^\vee$.

In \cite[Theorem~B]{How_duke}, Howard proved that when $p$ is a prime of good ordinary reduction of $\EC$ and the Heegner point $z_{\textrm{Heeg}}$ of $\EC$ over $K$ is non-torsion, the Selmer group $\Sel_{p^\infty}(\EC/K_{\ac})$ has $\Lambda_{\ac}$-corank 1.
Therefore, it follows from \cite[Corollary~4.12]{Gre_PCMS} that $\corank_{\Zp}\Sel(\EC/K_{\ac,n})=p^n+O(1)$ for $n\gg0$.
Hence, we conclude that \ref{Mazur-conj} holds.
\end{proof}

We conclude this article with a remark on the supersingular case.

\begin{remark}\label{rk:ss}
Suppose in this remark that $\EC$ has good supersingular reduction at $p$ with $a_p(\EC)=0$.
One may attempt to bound the Selmer coranks of $\EC$ through four signed Selmer groups, namely $\Sel^{++}(\EC/K_\infty)$, $\Sel^{--}(\EC/K_\infty)$, $\Sel^{+-}(\EC/K_\infty)$ and $\Sel^{-+}(\EC/K_\infty)$, as constructed in \cite{kim14}.
These Selmer groups are related to D.~Loeffler's signed $p$-adic $L$-functions (which we denote by $L_p^{\star\bullet}(X,Y)$, $\star,\bullet\in\{+,-\}$) constructed in \cite{loeffler13} via signed Iwasawa main conjectures. Indeed, a link between the classical Selmer group and the signed Selmer groups has been studied in \cite{leisprung}.
    
It is possible to show that under appropriate hypotheses, the anti-cyclotomic specializations of $\Sel^{++}(\EC/K_\infty)$ and $\Sel^{--}(\EC/K_\infty)$ are of corank one over $\Lambda_\ac$, while the counterparts of $\Sel^{+-}(\EC/K_\infty)$ and $\Sel^{-+}(\EC/K_\infty)$ are cotorsion (see \cite[Theorem~A.12]{GHKL}).
The plus and minus Heegner points constructed in \cite{LongoVigni2} and \cite{castellawan} should allow us to calculate the cyclotomic derivatives of $L_p^{++}$ and $L_p^{--}$. Our proof of Proposition~\ref{prop:Sel-tor} would generalize to show that $\Sel^{++}(\EC/K_{a,b})^\vee$ and $\Sel^{--}(\EC/K_{a,b})^\vee$ are $\Lambda_{a,b}$-torsion for $(a:b)\neq(0:1)$.

However, the same proof does not apply to the mixed-signed Selmer groups. 
Instead, if $\star\ne\bullet$, we would only be able to show that the derivative of $\pi_{a,b}\left(L_p^{\star\bullet}\right)$ is a linear combination of $\displaystyle\frac{\partial L_p^{\star\bullet}}{\partial X}(0,0)$ and $\displaystyle\frac{\partial L_p^{\star\bullet}}{\partial Y}(0,0)$, which can be zero for certain $(a:b)\ne(0:1)$.
Consequently, this is not sufficient to deduce the validity of \ref{Mazur-conj}.
\end{remark}

\bibliographystyle{amsalpha}
\bibliography{references}

\end{document}